\documentclass[a4paper]{article}

\usepackage{amsmath,amsfonts,amssymb,amsthm}
\usepackage{graphicx, epsfig, float}
\usepackage{url}
\usepackage[latin1]{inputenc}

\title{{\large \textbf{Permanent properties of amenable, transitive and faithful actions}}}

\author{Soyoung Moon\footnote{This work is supported by Swiss NSF grant $20-126689/1$.}}
\date{\today}

\newcommand{\supp}{\textrm{supp}}

\newcommand{\Ker}{\textrm{Ker}}

\theoremstyle{plain}
\newtheorem{thm}{Theorem}
\newtheorem{prop}[thm]{Proposition}
\newtheorem{cor}[thm]{Corollary}
\newtheorem{lem}[thm]{Lemma}

\theoremstyle{definition}
\newtheorem{defn}{Definition}[section]
\newtheorem{rem}{Remark}[section]
\newtheorem{ex}{Example}[section]

\long\def\symbolfootnote[#1]#2{
\begingroup
\def\thefootnote{\fnsymbol{footnote}}\footnote[#1]{#2}
\endgroup}

\begin{document}

\maketitle

\smallskip

\begin{abstract}
We study hereditary properties of the class of countable groups admitting an amenable, transitive and faithful action on a countable set. We consider mainly the case of amalgamated free products, and we show in particular that the double of amenable groups and the amalgamated free products of two amenable groups over a finite subgroup admit such actions.
\end{abstract}

\bigskip

\pagestyle{headings}

\section{Introduction}

Let $G$ be a countable group acting on a countable set $X$.

\begin{defn}
The action of $G$ on $X$ is \textit{amenable} if there exists a $G$-invariant mean on $X$, i.e. a map $\mu : 2^X=\mathcal{P}(X)
\rightarrow [0,1]$ such that $\mu(X)=1 $, $\mu (A \cup B)= \mu(A)+\mu(B)$ for every pair of disjoint subsets $A$, $B$ of $X$, and
$\mu(gA)=\mu(A)$, $\forall g \in G$, $\forall A \subseteq X$.
\end{defn}
A group is amenable if the action on itself by left multiplication is amenable.

\smallskip

The above definition is equivalent to the existence of a \textit{F\o lner sequence}, i.e., a sequence $\{A_n\}_{n\geq 1}$ of finite non-empty subsets of $X$ such that for every $g\in G$, one has
$$ \lim_{n\rightarrow \infty}\frac{|A_n \vartriangle g\cdot A_n|}{|A_n|}=0. $$

The class $\mathcal{A}$ of countable groups admitting an amenable, transitive and faithful action on a countable set is introduced by Glasner and Monod in \cite{GlMo}.
 Whilst the class $\mathcal{A}$ is closed under direct products, free products and extension of co-amenable subgroups\symbolfootnote[2]{A subgroup $H<G$ is \textit{co-amenable} in $G$ if the $G$-action on $G/H$ is amenable.} (Proposition 1.7. in \cite{GlMo}), it is not closed under amalgamated free products even if two factors are amenable groups. Indeed, viewing the group $SL_2(\mathbb{Z})$ as an amalgamated free product
$$
SL_2(\mathbb{Z})\simeq (\mathbb{Z}/6\mathbb{Z})\ast_{(\mathbb{Z}/2\mathbb{Z})}(\mathbb{Z}/4\mathbb{Z}),
$$
one can see the group $SL_2(\mathbb{Z}) \ltimes \mathbb{Z}^2$ as the amalgamated free product $G\ast_A H$ of $G=\mathbb{Z}/4\mathbb{Z} \ltimes \mathbb{Z}^2$ and $H=\mathbb{Z}/6\mathbb{Z} \ltimes \mathbb{Z}^2$ along $A=\mathbb{Z}/2\mathbb{Z} \ltimes \mathbb{Z}^2$. Notice that the three groups $G$, $H$ and $A$ are in $\mathcal{A}$ since they are amenable, but the amalgamated free product is not in $\mathcal{A}$ since the pair $(SL_2(\mathbb{Z}) \ltimes \mathbb{Z}^2, \mathbb{Z}^2)$ has relative Property (T) and $\mathbb{Z}^2$ does not have finite exponent (it follow from Lemma 4.3. in \cite{GlMo} that if $G\in \mathcal{A}$ and the pair $(G,H)$ has relative Property (T), then $H$ has finite exponent, i.e. there is an integer $n$ such that for every $h\in H$, $h^n=1$).

\smallskip

The purpose of this paper is to investigate some hereditary properties of the class $\mathcal{A}$ in case of amalgamated free products. Our first result is (see Proposition \ref{DoubleAmen}):
\begin{thm}\label{thm1}
Let $G$ be a group such that $G$ surjects onto an amenable group $H$. If $A$ is a common subgroup of $G$ and $H$ such that $\pi|_A$ is injective and $[H:A]\geq 2$, then the amalgamated free product $G \ast_A H$ is in $\mathcal{A}$. In particular, if $G$ is amenable, the double of $G$ over $A$ is in $\mathcal{A}$, for every subgroup $A$ of $G$.
\end{thm}

In case of double of free groups on two generators over a cyclic subgroup $\mathbb{F}_2 \ast_{\mathbb{Z}}\mathbb{F}_2$, it was proven in \cite{Moon} that any finite index subgroup of such group is in $\mathcal{A}$. Let us mention that this result was generalized to amalgamated free products of any free groups over a cyclic subgroup (called cyclically pinched one-relator groups) $\mathbb{F}_n \ast_{\mathbb{Z}}\mathbb{F}_m$, $\forall n$, $m\geq 2$ in \cite{Moon2}.

\smallskip
Our second theorem concerns amalgamated free products over a finite subgroup (see Proposition \ref{A'}). In particular we show:

\begin{thm}\label{thm2}
Let $G$ and $H$ be countable groups and $A$ be a common finite subgroup of $G$ and $H$. If $G$ is an infinite amenable group and there is a $H$-set $Y$ such that the $H$-action is amenable and the action of $A$ on $Y$ is free, then $G\ast_A H$ is in $\mathcal{A}$.
\end{thm}

As corollaries, the amalgamated free product of an infinite amenable group and a residually finite group over a finite subgroup (for instance $SL_3(\mathbb{Z})\ast_A G$ with any infinite amenable group $G$ and a common finite subgroup $A$) is in $\mathcal{A}$ (Corollary \ref{AmenResid}); and the amalgamated free product of two amenable groups over a finite subgroup is also contained in $\mathcal{A}$ (Corollary \ref{Amalgam-finite-sg}).
\bigskip

\noindent \textbf{Acknowledgement.} I would like to thank Nicolas Monod for useful discussions and Alain Valette for his valuable comments on a previous version of this paper.

\section{Hereditary properties of $\mathcal{A}$}

\subsection{Double of amenable groups}
An amalgamated free product $G\ast_A H$ is called \textit{double of $G$ over $A$} if $H$ is isomorphic to $G$ and the amalgamated subgroup of $H$ is given by the isomorphism. Such a group has a presentation $\langle G, \phi(G)|a=\phi(a)$, $\forall a\in A\rangle$ where $\phi:G\rightarrow H$ is an isomorphism. As mentioned above, in general the amalgamated free product is not in $\mathcal{A}$ even if the factors are amenable groups. But it is true if the amalgam is a double; or more generally, if one factor surjects onto the other factor which is amenable with some extra condition (Theorem \ref{thm1}):

\begin{prop}\label{DoubleAmen}
Let $H$ be an amenable group and let $\pi: G \twoheadrightarrow H$ be a group epimorphism and let $A<G$ be a subgroup such that $\pi|_A$ is injective and $[H:\pi(A)]\geq 2$. Then the amalgamated free product $G \ast_A H$ given by $\pi$, i.e.
$$
G\ast_A H =\langle G, H | \pi(a)=a, \forall a\in A\rangle
$$
is in $\mathcal{A}$. In particular, if $G$ is amenable, the double of $G$ over $A$ is in $\mathcal{A}$.
\end{prop}

\begin{proof}
Let $\psi: G\ast_A H\rightarrow H $ be a homomorphism defined by $\psi(g)=\pi(g)$, and $\psi(h)=h$ for all $g\in G$ and $h\in H$. Let $\Ker(\psi)=N$. Let $X$ be the Bass-Serre tree of $G \ast_A H$ and $Y=N\setminus X$ be the quotient graph. Since $N$ intersects $A$ trivially, by the theory of Bass-Serre, $N$ is isomorphic to the free product
$$ N \simeq K\ast (\ast_i H_i),$$
where $K$ is a free group isomorphic to the fundamental group $\pi_1(Y,T)$ of the graph $Y$ relatively to the maximal tree $T$ which is generated by $g_y$ with $y\in O - T$ where $O$ is an orientation of $Y$, and $H_i$ is the intersection of $N$ with a conjugate of some factor of $G\ast_A H$ (cf. Remarque in p. 61 in \cite{serre}). Thus in order to have $K\neq \{1\}$, it suffices that the quotient graph $Y=N\setminus X$ is not a tree. If there exist $x\in G$ and $h\in N$ such that $hH=xH$ and $hA\neq xA$ (see Figure \ref{NotQuotientTree}), then the quotient graph $Y$ will have a circuit of length 2.

\begin{figure}
\centering
\includegraphics[width=3cm]{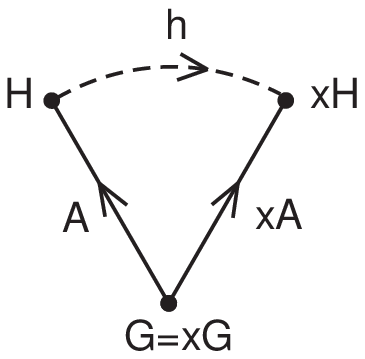}
\caption{}
\label{NotQuotientTree}
\end{figure}

Thus, it is sufficient to find $x\in G$ and $h\in N$ such that $x\in G\setminus A$ and $x^{-1}h\in H\setminus A$.

By assumption on the index of $\pi(A)$ in $H$, there exists $z\in H\setminus \pi(A)$ so that $z^{-1}\in H\setminus \pi(A)$. Let $x\in G\setminus A$ such that $\pi(x)=z$. Let $h=x\pi(x^{-1})$. Then $\psi(h)=\psi(x)\psi(\pi(x^{-1}))=\pi(x)\pi(x^{-1})=1$, so $h\in N$, and we have
$$x^{-1}h=x^{-1}x\pi(x^{-1})=\pi(x^{-1})=z^{-1}\in H\setminus A.$$

Therefore, the normal subgroup $N$ is a free product $K\ast (\ast_i H_i)$ with $K\neq \{1\}$, and the quotient is the amenable group $H$. Recall that a non-trivial free product $G_1\ast G_2$, where $G_i$ is a free group for some $i\in \{1,2\}$, is always in $\mathcal{A}$ from Theorem 1.5 in \cite{GlMo}. Recall also that the class $\mathcal{A}$ is closed under extension of co-amenable subgroups. It follows that $N$ is co-amenable in $G\ast_A H$ and is in $\mathcal{A}$, thus $G \ast_A H$ is in $\mathcal{A}$. For the second statement, take $\pi=\textrm{Id}_G$.
\end{proof}

\subsection{Amalgamated free products over a finite subgroup}

\begin{defn}\label{defA'}
Let $G$, $H$ be two countable groups and let $A$ be a common finite subgroup of $G$ and $H$. We say that the triple $(G,H,A)$ is in the class $\mathcal{A}'$ if there exist a $G$-action on $X$ and a $H$-action on $Y$ such that

\begin{enumerate}
\item[(i)] the action $G \curvearrowright X$ is transitive;
\item[(ii)] for every element $g$ of $G\setminus A$, and $h$ of $H\setminus A$, the sets
$$\supp_A(g)=\{x\in X| Ax\cap gAx=\emptyset\}$$
$$\supp_A(h)=\{x\in Y| Ax\cap hAx=\emptyset\}$$
are infinite;
\item[(iii)] there exist F\o lner sequences $\{C_n\}_{n\geq 1}$ of $G \curvearrowright X$ and $\{D_n\}_{n\geq 1}$ of $H \curvearrowright Y$ such that

(iii)-1. $|C_n|=|D_n|$, $\forall n\geq 1$;

(iii)-2. the sets $\{A\cdot C_n\}_{n\geq 1}$, $\{A\cdot D_n\}_{n\geq 1}$ are pairwise disjoint;

\item[(iv)] the action of $A$ on $X$ and $Y$ are free.
\end{enumerate}
\end{defn}
Note that if $(G,H,A)\in \mathcal{A}'$ then $G\in \mathcal{A}$.

\begin{prop}\label{A'}
If $(G,H,A)\in \mathcal{A}'$ then $G\ast_A H \in \mathcal{A}$.
\end{prop}

\begin{proof}
Let $X$ be a countable set carrying actions of $G$ and $H$ as in the definition. Since $A$ acts freely in the two actions, conjugating we may assume that the $G$-action and the $H$-action coincide on $A$.

First of all, let us recall that the group $Sym(X)$ of permutations of $X$ endowed with the topology of pointwise convergence\symbolfootnote[3]{A sequence $\alpha_n\in Sym(X)$ converges to $\alpha\in Sym(X)$ if for all finite subset $F$ of $X$, there exists $n_0$ such that $\alpha_n|_F=\alpha|_F$, for all $n\geq n_0$.} is a Baire space. Recall that a subset $Y\subset Sym(X)$ is \textit{meagre} if it is a union of countably many closed subsets with empty interior; and \textit{generic} or \textit{dense $G_{\delta}$} if its complement $Sym(X)\setminus Y$ is meagre. Baire's theorem states that in a complete metric space, the intersection of countably many dense open subsets is dense, in particular not empty. Thus in order to find a permutation $\alpha$ of $X$ having the properties $\{P_i\}_{i\geq 1}$, it is enough to prove that the set
$$\mathcal{U}_i=\{\alpha\in Sym(X)|\textrm{$\alpha$ satisfies the property $P_i$} \}$$
is generic in $Sym(X)$ and take $\alpha\in \cap_{i\geq 1}\mathcal{U}_i$.

\bigskip

Now let $\{C_n\}_{n\geq 1}$ (respectively $\{D_n\}_{n\geq 1}$) be the F\o lner sequence for $G$ (respectively for $H$) satisfying the condition (iii) as in Definition \ref{defA'}. Set $Z=\{\sigma\in Sym(X)$ $| \sigma a=a\sigma$, $\forall a\in A \}$; this is a closed subset in $Sym(X)$ so it is a Baire space. For $\sigma\in Z$, denote by $H^{\sigma}=\sigma^{-1}H\sigma$. By universality, the amalgamated free product $G\ast_A H^{\sigma}$ acts on $X$ by $g\cdot x=gx$ and $h\cdot x=\sigma^{-1}h\sigma x$ for all $g\in G$ and $h\in H$. We shall prove that the sets
$$
\mathcal{O}_1=\{\sigma\in Z | \textrm{ the action } G\ast_A H^{\sigma} \textrm{ on $X$ is faithful }\},
$$
and
$$
\mathcal{O}_2=\{\sigma \in Z| \textrm{ there is a subsequence $\{n_k\}$ of $n$ such that $\sigma (C_{n_k})=D_{n_k}$}\}
$$
are generic in $Z$.

Indeed, for the genericity of $\mathcal{O}_1$, we shall prove that for every non-trivial word $w\in G\ast_A H$, the set
$$\mathcal{V}_w=\{\sigma\in Z | w^{\sigma}=\textrm{Id}_X\}$$
 is closed and of empty interior. It is clear that the set $\mathcal{V}_w$ is closed. To prove that the set $\mathcal{V}_w$ is of empty interior, let us consider the case where $w=ag_nh_n\cdots g_1h_1$ with $a\in A$, $g_i\in G\setminus A$ and $h_i\in H\setminus A$ (the other three cases are similar). The corresponding element of $Sym(X)$ given by the action is $w^{\sigma}=ag_n\sigma^{-1}h_n\sigma \cdots g_1\sigma^{-1}h_1\sigma$. Let $\sigma\in \mathcal{V}_w$. Let $F\subset X$ be a finite subset. Choose $x_0\notin F\cup \sigma(F)$ such that $Ax_0\cap(F\cup\sigma(F))=\emptyset$. Inductively on $1\leq i\leq n$, we choose a new point $x_{4i-3}\in \supp_A(h_i)$ such that $Ax_{4i-3}$ and $h_i Ax_{4i-3}$ are outside of the finite set of all points defined before (this is possible by (ii) in Definition \ref{defA'}). Then we define
$$
\sigma'(ax_{4i-4}):=ax_{4i-3}\textrm{ and } \sigma'(a\sigma^{-1}(x_{4i-3})):=a\sigma(x_{4i-4}),
$$
for all $a\in A$. Then set $x_{4i-2}:=h_ix_{4i-3}$. We choose again a new point $x_{4i-1}\in \supp_A(g_i)$ such that $Ax_{4i-1}$ and $g_iAx_{4i-1}$ are outside of the finite set of all points considered so far. We then define
$$
\sigma'(ax_{4i-1}):=ax_{4i-2} \textrm{ and } \sigma'(a\sigma^{-1}(x_{4i-2})):=a\sigma(x_{4i-1}),
$$
for all $a\in A$. Then set $x_{4i}:=g_ix_{4i-1}$.

Every point $v$ on which $\sigma'$ is defined verifies $\sigma'a(v)=a\sigma'(v)$, $\forall a\in A$. Indeed, let $a$, $a'\in A$. Then,
\begin{enumerate}
\item[$\cdot$] $\sigma'a(a'x_{4i-4}) = \sigma'(aa'x_{4i-4}) = aa' x_{4i-3} = a(a' x_{4i-3}) = a\sigma'(a'x_{4i-4})$;
\item[$\cdot$] $\sigma'a(a'\sigma^{-1}(x_{4i-3})) = \sigma'(aa'\sigma^{-1}(x_{4i-3})) = aa'\sigma(x_{4i-4}) = a(a'\sigma(x_{4i-4})) = a\sigma'(a'\sigma^{-1}(x_{4i-3}))$;
\item[$\cdot$] $\sigma'a(a'x_{4i-1}) = \sigma'(aa' x_{4i-1}) = aa' x_{4i-2} = a(a' x_{4i-2}) = a\sigma'(a'x_{4i-1})$;
\item[$\cdot$] $\sigma'a(a'\sigma^{-1}(x_{4i-2})) = \sigma'(aa'\sigma^{-1}(x_{4i-2})) = aa'\sigma(x_{4i-1}) = a (a'\sigma(x_{4i-1})) = a\sigma'(a'\sigma^{-1}(x_{4i-2}))$.
\end{enumerate}

By construction, the $4n+1$ points obtained by the right subwords of $w^{\sigma'}$ are all distinct and in particular $w^{\sigma'}x_0\neq x_0$. We then define $\sigma'$ to be $\sigma$ on every points except these finite points, so that $\sigma'\in Z\setminus \mathcal{V}_w$ satisfies $\sigma'|_F=\sigma|_F$. This concludes the genericity of $\mathcal{O}_1$.

\bigskip

About the genericity of $\mathcal{O}_2$, let us write $\mathcal{O}_2= \bigcap_{N\in \mathbb{N}} \{\sigma\in Sym(X)|$ there exists $m\geq N$ such that $\sigma(C_m)=D_m\}$. We shall show that for every $N\in \mathbb{N}$, the set $\mathcal{V}_N= \{\sigma \in Sym(X)| \forall m\geq N, \sigma(C_m)\neq D_m\}$ is of empty interior (the closedness is clear). Let $F\subset X$ be a finite subset and $\sigma\in \mathcal{V}_N$. Let $m\geq N$ large enough such that $A\cdot x_i\cap (F\cup \sigma(F))=\emptyset$ and $A\cdot y_i\cap (F\cup \sigma(F))=\emptyset$, for every $x_i\in C_m$ and $y_i\in D_m$, $\forall i$; this is possible by (iii)-2. of Definition \ref{defA'}. By (iv) of Definition \ref{defA'}, we have $|A\cdot x_i|=|A\cdot y_i|$, $\forall i$. We then define
$$
\sigma'(ax_i):=ay_i \textrm{ and } \sigma'(a\sigma^{-1}(y_i)):=a\sigma(x_i),
$$
for every $1\leq i\leq |C_m|=|D_m|$ and $a\in A$. For all other points, we define $\sigma'$ to be equal to $\sigma$ so that $\sigma'\in Z\setminus \mathcal{V}_{N}$ and $\sigma'|_F=\sigma|_F$.

\bigskip
Let $\sigma \in \mathcal{O}_1\cap \mathcal{O}_2$. Let $\{C_{n_k}\}_{k\geq 1}$ be the subsequence of $\{C_n\}_{n\geq 1}$ such that $\sigma(C_{n_k})=D_{n_k}$, $\forall k\geq 1$. The sequence $\{C_{n_k}\}_{k\geq 1}$ is a F\o lner sequence for $G$, and for every $h\in H$, we have

\begin{eqnarray}
\lim_{k\rightarrow \infty}\frac{|C_{n_k} \vartriangle h\cdot C_{n_k}|}{|C_{n_k}|}&=&\lim_{k\rightarrow \infty}\frac{|C_{n_k} \vartriangle \sigma^{-1}h\sigma C_{n_k}|}{|C_{n_k}|}= \lim_{k\rightarrow \infty}\frac{|\sigma C_{n_k} \vartriangle h\sigma C_{n_k}|}{|C_{n_k}|}\nonumber \\
&=& \lim_{k\rightarrow \infty}\frac{|D_{n_k} \vartriangle hD_{n_k}|}{|D_{n_k}|}=0.\nonumber
\end{eqnarray}
Thus the sequence $\{C_{n_k}\}_{k\geq 1}$ is a F\o lner sequence for $G\ast_A H^{\sigma}$, and therefore $G\ast_A H^{\sigma}$ is in $\mathcal{A}$.
\end{proof}

\begin{rem}\label{remA'}
\begin{enumerate}
\item The condition (ii) in Definition \ref{defA'} is trivially satisfied if the $G$-action on $X$ is free (which implies that $G$ is amenable).
\item About (iv) in Definition \ref{defA'}, this condition is used in the proof of the genericity of $\mathcal{O}_1$ and $\mathcal{O}_2$ where, given any two points $x$ and $y$, we needed to have $|Ax|=|Ay|$ in order to define $\sigma'$ such that $\sigma'(Ax)=Ay$.
\end{enumerate}
\end{rem}

The following lemma shows that given two infinite amenable groups, one can always find F\o lner sequences having the same cardinality. Precisely, we have:

\begin{lem}\label{cardFolner}
Let $G$, $H$ be infinite amenable groups. Then there exist F\o lner sequences $\{C_n\}_{n\geq 1}$ of $G$ and $\{D_n\}_{n\geq 1}$ of $H$ such that $|C_n|=|D_n|$, $\forall n\geq 1$.
\end{lem}

\begin{proof}
We shall show that for any $\varepsilon>0$, any finite subset $F\subset G$ and any finite subset $E\subset H$, there exist a finite subset $C'\subset G$ and a finite subset $D'\subset H$ such that $C'$ is $(\varepsilon, F)$-F\o lner and $D'$ is $(\varepsilon, E)$-F\o lner verifying $|C'|=|D'|$.

Recall that a finite subset $A\subset G$ is $(\varepsilon, F)$-F\o lner if
$$|A\vartriangle gA|< \varepsilon|A|, \,\, \forall g\in F.$$
By amenability of $G$, there is a finite subset $C_0\subset G$ which is $(\varepsilon,F)$-F\o lner. Let $\{D_n\}_{n\geq 1}$ be a F\o lner sequence of $H$. Let $n\gg 1$ large enough such that
\begin{enumerate}
\item[(1)] $D_n$ is $(\varepsilon /4, E)$-F\o lner;
\item[(2)] $|D_n|> \lambda |C_0|$, where $\lambda=$ max $\{8/ \varepsilon, 2 \}$.
\end{enumerate}
By Euclidean division, there exist $d$, $r\in \mathbb{N}$ such that $|D_n|=d|C_0|+r$ with $r<|C_0|$. Let $g_1$, $\dots$, $g_d\in G$ such that $\{C_0g_i\}_i$ are pairwise disjoint. We put
$C':=\bigsqcup_{i=1}^d C_0g_i$. Then $C'$ is $(\varepsilon, F)$-F\o lner and $|C'|=d|C_0|$.

Now let $D':=D_n - \{x_1,\dots, x_r\}$ be a subset of $D_n$ obtained from $D_n$ by deleting any $r$ elements of $D_n$. Then $|D'|=|D_n|-r=d|C_0|=|C'|$. We claim that $D'$ is $(\varepsilon,E)$-F\o lner. Indeed, we have
$$
\frac{|D_n|}{|D'|}=\frac{|D_n|}{d|C_0|}=\frac{d|C_0|+r}{d|C_0|}=1+\frac{r}{d|C_0|}<1+\frac{1}{d}\leq 2,
$$
since $r<|C_0|$ and $|D_n|=d|C_0|+r>\lambda|C_0|$ so $d>\lambda-1\geq 1$ by definition of $\lambda$. In addition, we have
$$
\frac{|D'\vartriangle D_n|}{|D_n|}=\frac{r}{|D_n|}<\frac{r}{\lambda|C_0|}<\frac{1}{\lambda}\leq \frac{\varepsilon}{8}.
$$
Therefore,
\begin{eqnarray}
\frac{|D'\vartriangle hD'|}{|D'|}& \leq &\frac{|D_n|}{|D'|}\Big(\frac{|D'\vartriangle D_n|}{|D_n|}+ \frac{|D_n\vartriangle hD_n|}{|D_n|}+\frac{|hD_n\vartriangle hD'|}{|D_n|}\Big) \nonumber \\
&< & 2\big(\frac{\varepsilon}{8}+\frac{\varepsilon}{4}+\frac{\varepsilon}{8}\big)=\varepsilon, \nonumber
\end{eqnarray}
for every $h\in E$.
\end{proof}

\begin{lem}\label{cardFolner2}
If there exist amenable actions of $G$ and $H$, then there exist $G$-action and $H$-action such that the actions admit F\o lner sequences $\{C_n\}_{n\geq 1}$ for the $G$-action and $\{D_n\}_{n\geq 1}$ for the $H$-action with $|C_n|=|D_n|$, $\forall n\geq 1$.
\end{lem}

\begin{proof}
If $G\curvearrowright X$ amenably, we replace $X$ by a disjoint union of infinitely many copies of $X$ and use the same idea as in the proof of Lemma \ref{cardFolner};  if $C_0\in X$ is $(\varepsilon, F)$-F\o lner, we put $C'$ to be $d$ copies of $C_0$ in $d$ disjoint copies of $X$.
\end{proof}

Now we are ready to prove Theorem \ref{thm2}:

\begin{cor}\label{GHA'}
Let $G$ and $H$ be countable groups and $A$ be a common finite subgroup of $G$ and $H$. If $G$ is an infinite amenable group and there is a $H$-set $Y$ such that the $H$-action is amenable and the action of $A$ on $Y$ is free, then $(G,H,A)\in \mathcal{A}'$.
\end{cor}

\begin{proof}
Let $\{C_n\}_{n\geq 1}$ be a F\o lner sequence of $G$. First of all, we can suppose that the sequence $\{A\cdot C_n\}_{n\geq 1}$ is pairwise disjoint. Indeed, if $\{C_n\}_{n\geq 1}$ is a F\o lner sequence of $G$, we define $\{C'_n\}_{n\geq 1}$ inductively on $n$; let $C'_1:=C_1$ and for every $n\geq 1$, we choose $h_n\in G$ such that $AC_{n+1}h_n\cap (\cup_{i=1}^n AC'_i)=\emptyset$, and we set $C'_{n+1}:=C_{n+1}h_n$. Moreover, if $\{D_n\}_{n\geq 1}\subset Y$ is a F\o lner sequence for $H$-action, we can also suppose that the sequence $\{A\cdot D_n\}_{n\geq 1}$ is pairwise disjoint. Indeed, let $Y_0=\sqcup_{i\geq 1} Y_i $ where $Y_i=Y$, $\forall i\geq 1$ be a disjoint union of infinitely many copies of $Y$. Then the $H$-action on $Y_0$ is amenable with a F\o lner sequence $D_n \subset Y_n$, $\forall n\geq 1$ such that $\{A\cdot D_n\}_{n\geq 1}$ is pairwise disjoint. In addition, by Lemma \ref{cardFolner2} we can suppose that $|C_n|=|D_n|$, $\forall n\geq 1$, so that the condition (iii) in Definition \ref{defA'} is verified.

Now we consider the $H$-action on $Y':= H\sqcup Y_0$. The action satisfies the condition (ii) of Definition \ref{defA'} (for the $H$-action) since the $H$-action on itself is free, and the action of $A$ on $Y'$ is free (since the $A$-action on $H$ is clearly free and the $A$-action on $Y_0$ is free by assumption). Finally let $G=X$. The $G$-action on itself is transitive and free. Thus the triple $(G,H,A)$ is in $\mathcal{A}'$.
\end{proof}

\begin{cor}\label{AmenResid}
If $G$ is infinite amenable group and $H$ contains a finite index normal subgroup $N$ with $N\cap A=\{1\}$, then $(G,H,A)\in \mathcal{A}'$.
\end{cor}
It follows from Corollary \ref{GHA'} by taking $Y=H/N$. So for example if $H$ is residually finite, applying Proposition \ref{A'} we see that $G\ast_A H$ is in $\mathcal{A}$ for every finite subgroup $A$ of $G$ and $H$. Besides, with $A=\{1\}$, we find a particular case of the result of Glasner-Monod; if $G$ is amenable, then $G\ast H\in\mathcal{A}$ for every countable group $H$.

\begin{cor}\label{Amalgam-finite-sg}
Let $G$, $H$ be amenable groups and let $A$ be a common finite subgroup of $G$ and $H$. Then the amalgamated free product $G\ast_A H$ is in $\mathcal{A}$.
\end{cor}

\begin{proof}
The cases where $G$ or $H$ is infinite follow from Corollary \ref{GHA'} and Proposition \ref{A'}. So let $G$ and $H$ be finite groups. From a result of Baumslag \cite{Baum}, if $G$ and $H$ are finite groups, then the amalgamated free product $G\ast_A H$ contains a free subgroup of finite index. So $G\ast_A H$ is in $\mathcal{A}$ since $\mathcal{A}$ is closed under the extension of co-amenable subgroup.
\end{proof}

\begin{rem}
When $G$ is a finitely generated group with polynomial growth, Lemma \ref{cardFolner} can be strengthened. Indeed, in this case $G$ admits F\o lner sequences of any prescribed size. More precisely, let $\{a_n\}_{n\geq 1}$ be a strictly ascending sequence of positive integers. Let $G$ be an infinite finitely generated group with polynomial growth. Then $G$ has a F\o lner sequence $\{F_n\}_{n\geq 1}$ such that $|F_n|=a_n$, $\forall n\geq 1$.
\end{rem}

\begin{proof}
Let $S$ be a finite symmetric generating set of $G$. Denote $B(k)$ the ball of radius $k$ centered at 1 in the Cayley graph $\mathcal{G}(G,S)$. Let $k_n$ such that $|B(k_n)|\leq a_n<|B(k_n+1)|$. We choose a finite subset $K_n$ such that $K_n\cap B(k_n)=\emptyset$ and $|K_n|=a_n-|B(k_n)|$, and set $F_n:=B(k_n)\cup K_n$, $\forall n\geq 1$. Recall that the boundary $\partial A$ of $A$ is the set of edges $(s,t)$ such that $s\in A$ and $t\notin A$. We have
\begin{eqnarray}
\frac{|\partial F_n|}{|F_n|} & \leq & \frac{|\partial B(k_n)|}{|F_n|}+\frac{|\partial K_n|}{|F_n|}\leq \frac{|\partial B(k_n)|}{|B(k_n)|}+\frac{|\partial K_n|}{|B(k_n)|} \nonumber \\
&\leq & \frac{|\partial B(k_n)|}{|B(k_n)|} +|S|\frac{|K_n|}{|B(k_n)|}\leq \frac{|\partial B(k_n)|}{|B(k_n)|}+|S|\frac{|B(k_n+1)|-|B(k_n)|}{|B(k_n)|} \nonumber \\
&\leq & \frac{|\partial B(k_n)|}{|B(k_n)|}+|S|\frac{|\partial B(k_n)|}{|B(k_n)|} \leq (1+|S|)\frac{|\partial B(k_n)|}{|B(k_n)|}. \nonumber
\end{eqnarray}
By a result of Pansu \cite{Pansu}, a group with polynomial growth with degree $d$ satisfies $\frac{|B(k)|}{k^d}\xrightarrow[k \rightarrow \infty]{}C$, for some $C>0$. Thus in such a group, the sequence of all balls is a F\o lner sequence. This applies to our case and thus $|\partial F_n|/|F_n|\xrightarrow[k \rightarrow \infty]{}0$.
\end{proof}

\subsection{Central extensions}

From the idea of Lemma 7.3.1 in \cite{CCJJV}, we have:
\begin{lem}\label{central}
Let $1\rightarrow Z \rightarrow G \xrightarrow[]{p} Q\rightarrow 1$ be a central extension. Suppose that there is a co-amenable subgroup $H<Q$ such that $H\in \mathcal{A}$ (so $Q$ is also in $\mathcal{A}$) and the central extension splits over $H$ (i.e. the central extension $1\rightarrow Z \rightarrow p^{-1}(H) \xrightarrow[]{p} H\rightarrow 1$ splits). Then $G\in \mathcal{A}$.
\end{lem}

\begin{proof}
Since the extension splits over $H$, the group $p^{-1}(H)$ is a semi-direct product of $Z$ and $H$, but $Z$ is central so $p^{-1}(H)$ is indeed isomorphic to the direct product $Z\times H$. Since $Z\in \mathcal{A}$ ($Z$ is amenable) and $H\in \mathcal{A}$ by assumption, the group $p^{-1}(H)$ is in $\mathcal{A}$. Moreover, the map $G/p^{-1}(H)\rightarrow Q/H$ defined by $gp^{-1}(H)\mapsto p(g)H$ is $G$-equivariant and bijective, so the co-amenability of $H$ in $Q$ implies the co-amenability of $p^{-1}(H)$ in $G$, thus $G\in \mathcal{A}$.
\end{proof}

\begin{ex}(Example 7.3.4. in \cite{CCJJV})\label{torus}
For $p$, $q\geq 2$, a \textit{Torus knot group} $\Gamma_{p,q}$ is the group with the presentation $\Gamma_{p,q}=\langle x,y|x^p=y^q\rangle$.
There is a central extension
$$
0\rightarrow\mathbb{Z}\rightarrow \Gamma_{p,q}\rightarrow \mathbb{Z}/p\mathbb{Z}\ast \mathbb{Z}/q\mathbb{Z}\rightarrow 1,
$$
where $\mathbb{Z}\simeq \langle x^p=y^q\rangle$, $\mathbb{Z}/p\mathbb{Z}\simeq \langle x|x^p=1\rangle$ and $\mathbb{Z}/q\mathbb{Z}\simeq \langle y|y^q=1\rangle$.

The free product $\mathbb{Z}/p\mathbb{Z}\ast \mathbb{Z}/q\mathbb{Z}$ has a finite index free subgroup $\mathbb{F}$ over which the central extension splits. Thus the group $\Gamma_{p,q}$ is in $\mathcal{A}$ by Lemma \ref{central}.
\end{ex}

\begin{ex}
Let $M$ be a 3-manifold which is constructed as a fiber bundle over a closed orientable surface with fiber a circle. Such a 3-manifold is an orientable \textit{Seifert fibred space}. There is a central extension
$$
0\rightarrow \mathbb{Z} \rightarrow \pi_1(M)\xrightarrow[]{p} \Gamma_g \rightarrow 1,
$$
where $\Gamma_g$ is the fundamental group of the closed orientable surface of genus $g\geq 2$. The derived subgroup $\Gamma'_g$ is co-amenable in $\Gamma_g$ and $\Gamma'_g \in \mathcal{A}$ since $\Gamma'_g$ is free (since $[\Gamma_g: \Gamma'_g]$ is infinite). Thus we have a central extension that splits:
$$
1\rightarrow \mathbb{Z}\rightarrow p^{-1}(\Gamma'_g)\xrightarrow[]{p} \Gamma'_g\rightarrow 1.
$$
Therefore $\pi_1(M)$ is in $\mathcal{A}$ by Lemma \ref{central}.
\end{ex}

\bibliographystyle{amsplain}
\bibliography{ThesisBib}

\smallskip

\begin{quote}
Soyoung \textsc{Moon} \\
Institut de Math\'ematiques \\
Universit\'e de Neuch\^atel\\
11, Rue Emile Argand - BP 158\\
2009 Neuch\^atel - Switzerland

\smallskip

E-mail: \url{so.moon@unine.ch}\\
Homepage: \url{http://members.unine.ch/so.moon}
\end{quote}
\end{document}